\documentclass[11pt]{amsart} 
\usepackage{amssymb,amsmath,latexsym,enumerate,graphicx,bbm,mathptmx,microtype,cite,xcolor,tikz, array, enumitem, adjustbox}
\usepackage[colorinlistoftodos]{todonotes}
\usepackage{tikz}
\definecolor{munsell}{rgb}{0.0, 0.5, 0.69}

\allowdisplaybreaks 

\usepackage{cmbright}

\usepackage[hmargin=1.25in, vmargin=1.5in]{geometry} 

\usepackage{hyperref}
    \hypersetup{%
        colorlinks=true,
        linkcolor=blue,
        citecolor=magenta,
        filecolor=magenta,      
        urlcolor=magenta,
        bookmarksopen=true,   
    }%

\newtheorem{theorem}{Theorem} 
\newtheorem{corollary}[theorem]{Corollary}

\newtheorem{lemma}[theorem]{Lemma}
\newtheorem{conjecture}{Conjecture}
\newtheorem{exam}{Example}

\newenvironment{example}{\begin{exam}\sf}{\end{exam}}
\newtheorem{rem}{Remark}
\newenvironment{remark}{\begin{rem}\sf}{\end{rem}}

\newcommand\commentout[1]{}

\newcommand\qbinom[2]{\left[{#1} \atop {#2} \right]_q}

\newcommand\ehr{\operatorname{ehr}} 
\newcommand\Ehr{\operatorname{Ehr}}

\newcommand\Des{\operatorname{Des}}
\newcommand\des{\operatorname{des}}
\newcommand\Asc{\operatorname{Asc}}
\newcommand\asc{\operatorname{asc}}
\newcommand\add{\operatorname{add}}
\newcommand\maj{\operatorname{maj}}
\newcommand\comaj{\operatorname{comaj}}
\newcommand\exc{\operatorname{exc}}

\newcommand\ZZ{\mathbb{Z}}
\newcommand\QQ{\mathbb{Q}}
\newcommand\RR{\mathbb{R}}

\newcommand\bm{\mathbf{m}}

\newcommand\bone{\mathbf{1}}
\newcommand\op{^{\rm op}}

\makeatletter 
\newtheorem*{rep@theorem}{\rep@title}\newcommand{\newreptheorem}[2]{%
\newenvironment{rep#1}[1]{%
\def\rep@title{\bf #2 \ref{##1}}%
\begin{rep@theorem}}%
{\end{rep@theorem}}}
\makeatother
\newreptheorem{theorem}{Theorem}

\begin{document}

\title{$q$-chromatic polynomials}

\author{Esme Bajo}
\address{Department of Mathematics, San Diego Miramar College}
\email{ebajo@sdccd.edu}

\author{Matthias Beck}
\address{Department of Mathematics, San Francisco State University}
\email{mattbeck@sfsu.edu}

\author{Andr\'{e}s R. Vindas-Mel\'{e}ndez*}
\address{Department of Mathematics\\
         Harvey Mudd College}
\email{avindasmelendez@g.hmc.edu}
\thanks{*Corresponding author}

\begin{abstract}
We study a $q$-version of the chromatic polynomial of a given graph $G=(V,E)$, namely,
\[
\chi_G^\lambda(q,n) \ := \sum_{\substack{\text{proper colorings}\\ c\,:\,V\to[n]}} q^{ \sum_{ v \in V } \lambda_v c(v) },
\]
where $\lambda \in \ZZ_{>0}^V$ is a fixed linear form.
Via work of Chapoton (2016) on $q$-Ehrhart polynomials, $\chi_G^\lambda(q,n)$ turns out to be a polynomial in the $q$-integer $[n]_q$, with coefficients that are rational functions in $q$. 
Additionally, we prove structural results for $\chi_G^\lambda(q,n)$ and exhibit connections to neighboring concepts, e.g., chromatic symmetric functions and the arithmetic of order polytopes. 
We offer a strengthened version of Stanley's conjecture that the chromatic symmetric function distinguishes trees, which leads to an analogue of $P$-partitions for graphs. 
\end{abstract}


\maketitle


\section{Introduction}

The \emph{chromatic polynomial} of a graph $G=(V,E)$, 
\[
\chi_G(n) \ := \ \#\{c:V\to[n] \, : \, c(v)\neq c(w) \text{ if } vw \in E\} \, ,
\]
where $[n] := \{ 1, 2, \dots, n \}$, is a famous and much-studied enumerative invariant of~$G$. 
We study the following refinement: given $\lambda \in \ZZ_{>0}^V$, let 
\[
\chi_G^\lambda(q,n) \ := \sum_{\substack{\text{proper colorings}\\ c\,:\,V\to[n]}} q^{ \sum_{ v \in V } \lambda_v c(v) }.
\]
Naturally, $\chi_G^\lambda(1,n) = \chi_G(n)$.
The (crucial) special $\lambda = \bone \in \ZZ^V$, i.e., $\lambda$ is a vector whose entries are all~1, was previously studied by Loebl~\cite{loebl}, using the color set $\{0,1,\dots, n-1\}$. 
If $M_q(G,n)$ denotes his enumeration function, then $\chi_G^\bone(q,n)=q^{|V|}M_q(G,n)$. 
We present new alternative formulations of this enumeration function with our color convention. 
Our slightly more general definition mirrors Chapoton's study of $q$-Ehrhart polynomials~\cite{Chapoton}; it
turns out Chapoton's viewpoint gives a crucial polynomiality structure, which we employ below. 

On the other hand, consider Stanley's \emph{chromatic symmetric function}~\cite{StanleyChromatic}
\[
X_G(x_1,x_2,\hdots) \ :=\sum_{\substack{\text{proper colorings}\\ c:V\to \ZZ_{
> 0 }}} x_1^{\# c^{-1}(1)}x_2^{\# c^{-1}(2)}\cdots
\]
(so that $X_G(\underbrace{1,1,\dots,1}_{\text{$n$ 1's}}, 0, 0, \dots) = \chi_G(n)$\,).
Its \emph{principal evaluation} (sometimes referred to as the \emph{principal specialization})
\begin{equation}\label{eq:symmfcteval}
X_G(q,q^2,\hdots,q^n,0,0,\hdots)
  \ =\sum_{\substack{\text{proper colorings}\\ c:V\to [n]}} q^{\sum_{v \in V} c(v)}
  \ = \ \chi_G^\bone(q,n) \, .
\end{equation}
We think of $X_G(x_1,x_2,\hdots)$ and $\chi_G^\lambda(q,n)$ as (quite) different generalizations of the chromatic polynomial, which meet in~\eqref{eq:symmfcteval} and still generalize $\chi_G(n)$.
We remark that $\chi_G^\lambda(q,n)$ depends on the indexing of the nodes of $G$, whereas $\chi_G^\bone(q,n)$ does not. 

Our first result says that $\chi_G^\lambda(q,n)$ has a polynomial structure whose coefficients are rational functions in $q$, in the following sense:

\begin{theorem}\label{thm:chitildepoly}
There exists a unique polynomial $\widetilde{\chi}_G^\lambda(q,x)\in\QQ(q)[x]$ such that
\[
\widetilde{\chi}_G^\lambda(q,[n]_q) \ = \ \chi_G^\lambda(q,n),
\]
where $[n]_q := \frac{ 1-q^n }{ 1-q }$. 
\end{theorem}

We thus call $\widetilde{\chi}_G^\lambda(q,x)$ (and sometimes, by a slight abuse of nomenclature, $\chi_G^\lambda(q,n)$) the \emph{$q$-chromatic polynomial} of $G$ with respect to~$\lambda$.
Our main goal is to initiate the study of this polynomial.
Again, the case $\lambda = \bone$ goes back to Loebl~\cite{loebl}, one of whose results is described in
Remark~\ref{rem:loebl} below, which shows the polynomiality predicted by Theorem~\ref{thm:chitildepoly} for this case; however, he did not study the polynomials $\widetilde{\chi}_G^\bone(q,x)$.

\begin{example} \label{ex:path}
Consider the path $P_2$ with $2$ vertices. The following table shows $\widetilde{\chi}_G^\lambda(q,x)$ and  $ \chi_G^\lambda(q,n)$ for $\lambda = (1,1)$ and $(1,2)$.
\begin{center}
\noindent\adjustbox{max width=\textwidth}{
\begin{tabular}{ |c|c|c| } 
 \hline 
$\lambda$ & $\widetilde{\chi}_{P_2}^\lambda(q,x)$ &  $ \chi_{P_2}^\lambda(q,n)$  \\ 
 \hline \hline
 (1,1) & $\displaystyle{\frac{2q^2}{q + 1}x^2 - \frac{2q^2}{q + 1}x}$  & $\displaystyle{ q^2\left(\left(\frac{1-q^n}{1-q}\right)^2 - \left(\frac{1-q^{2n}}{1-q^2}\right)\right)}$ \\ 
 \hline
 (1,2) & $\displaystyle{\frac{q^5+q^4-2q^3}{q^3+2q^2+2q+1}x^3+\frac{-q^5+2q^4+5q^3}{q^3+2q^2+2q+1}x^2-\frac{3q^3}{q^2+q+1}x}$ &$\displaystyle{q^3\left(\frac{1-q^n}{1-q}\,\frac{1-q^{2n}}{1-q^2}-\frac{1-q^{3n}}{1-q^3}\right)}$\\
 \hline 
 \end{tabular}}
\end{center}
\noindent Note that the chromatic polynomial $\chi_{P_2}(n)=n^2-n$ appears for $q=1$.
\end{example}

There are several motivations to study $\chi_G^\lambda(q,n)$ and $\widetilde{\chi}_G^\lambda(q,x)$.
We already mentioned Chapoton's work on $q$-Ehrhart polynomials~\cite{Chapoton} and, in fact, Theorem~\ref{thm:chitildepoly} follows from Chapoton's work and the interplay of chromatic and order polynomials, as we will show in Section~\ref{sec:qehrhart} below.
On the graph-theoretic side, Stanley famously conjectured that $X_G(x_1,x_2,\hdots)$
distinguishes trees; this conjecture has been checked for trees with $\le 29$
vertices~\cite{heilji}, but remains open in general. For recent progress on this conjecture,
see, e.g., \cite{aliniaeifardvanWilligenburg, alisteprietoetal, Alistedemieretal} and the references therein.
The literature contains several variations of Stanley's chromatic symmetric function; some references on those different variations include~\cite{gebhardsagan, shareshianwachs, hasebetsujie,pawlowski, alisteprietoetal, Alistedemieretal}.
We particularly point out recent work of Crew and Spirkl~\cite{crewspirkl} 
(a variant in terms of graph polynomials can already be found in~\cite{noblewelsh}, which in turn
was motivated by~\cite{chmutovduzhinlando}),
who introduced a weighted form of the chromatic symmetric function (and so $\chi_G^\lambda(q,n)$ is a special evaluation, with the weights given by $\lambda$)
and of Loehr and Warrington~\cite{loehrwarrington} who conjectured, more strongly, that the principal evaluation~\eqref{eq:symmfcteval} distinguishes trees; they confirmed this conjecture for all trees with $\le 17$ vertices.
We offer the following further strengthening, which we have checked for all trees with $\le 17$ vertices. 

\begin{conjecture}\label{conj:main}
The leading coefficient of the $q$-chromatic polynomial $\widetilde{\chi}_G^\bone(q,x)$ distinguishes trees.
\end{conjecture}

From our construction of $\widetilde{\chi}^\lambda_G(q,x)$ for general $\lambda$, we offered (in a pre-version
of this article) the natural
weakening of Conjecture~\ref{conj:main} suggesting that for any pair of non-isomorphic trees $S$ and $T$ on $d$
vertices, there exists a vector $\lambda=(\lambda_1,\hdots,\lambda_d)\in \ZZ_{>0}^d$ such that $\chi_S^\lambda(q,n)\neq \chi_T^\lambda(q,n)$.
This weaker conjecture was recently proved~\cite{generatingfunctionsqchromaticpolynomials}; it holds on the
level of graphs (not just trees), and $\lambda$ depends only on~$d$. 


It turns out (as we will see in Corollary~\ref{cor:leadingtostable}) that knowing the leading coefficient of
the $q$-chromatic polynomial $\widetilde{\chi}_G^\bone(q,x)$ is equivalent to knowing the \emph{stable
principal evaluation} $X_G(q,q^2,q^3,\hdots)$ of the chromatic symmetric function, and so in this sense
Conjecture~\ref{conj:main} is not new, but we hope it can provide a new avenue to attack various versions of
Stanley's tree conjecture.
In fact, the stable principal evaluation of the chromatic symmetric function also appears as the special
evaluation $\widetilde{\chi}_G^\bone ( q, \frac{ 1 }{ 1-q })$ of the $q$-chromatic polynomial
(Theorem~\ref{thm:oneoveroneminusq}).
In the language of chromatic symmetric functions, this exemplifies that the conjectures that
$X_G(x_1,x_2,x_3,\hdots)$, 
$X_G(q,q^2,\hdots,q^n,0,0,\hdots)$, and
$X_G(q,q^2,q^3,\hdots)$
can respectively distinguish trees are in strict order of strength.

Section~\ref{sec:structure} contains several further structural results for $q$-chromatic polynomials: deletion--contraction musings (Theorems \ref{thm:deletion-contraction} and \ref{thm:repeated_application}), a combinatorial reciprocity theorem (Theorem \ref{thm:qchromrec}), and a formula for $\widetilde{\chi}_G^\lambda(q,x)$ in terms of the M\"obius function of the flats of the given graph (Theorem~\ref{thm:mobiusformula}).
We mostly concentrate on results on the polynomial $\widetilde{\chi}_G^\lambda(q,x)$; 
there are further structural results on the enumeration function $\chi_G^\lambda(q,n)$ that
are direct consequences of their counterparts on the (weighted) chromatic symmetric function
side.

In Section~\ref{sec:leadingcoeff} we give several formulas for $\widetilde{\chi}_G^\bone(q,x)$.
One of them naturally suggests an analogue of Stanley's $P$-partitions~\cite{stanleythesis}, moving from posets to graphs:  we introduce and study $G$-partitions in Section~\ref{sec:gpartitions} and show that Conjecture~\ref{conj:main} is equivalent to saying that $G$-partitions distinguish trees.
Furthermore, we exhibit a connection between $G$-partitions and the stable principal evaluation of the chromatic symmetric function.


\section{$q$-Ehrhart Polynomials}\label{sec:qehrhart}

Chapoton~\cite{Chapoton} introduced a weighted generalization of the Ehrhart polynomial of a \emph{lattice polytope} $P\subseteq \RR^d$ (i.e., $P$ is the convex hull of finitely many integer lattice points in $\ZZ^d$). 
We briefly sketch this theory and its application to order polytopes, which in turn allows us to exhibit a connection to $q$-chromatic polynomials.

Let $\lambda:\ZZ^d\to\ZZ$ be a linear form that is nonnegative on the vertices of $P$ and
distinguishes the vertices of any edge of $P$, and define 
\[
\ehr_P^\lambda(q,n) \ := \sum_{m \in nP\cap\ZZ^d} q^{\lambda(m)}.
\]
The classical Ehrhart polynomial \cite{ehrhartpolynomial} is the specialization $\ehr_P^\lambda(1,n)$. 
Chapoton~\cite{Chapoton} proved that there is a polynomial $\widetilde{\ehr}_P^\lambda(x)\in \QQ(q)[x]$, such that
\begin{equation}\label{eq:chapotonqehr}
\widetilde{\ehr}_P^\lambda([n]_q) \ = \ \ehr_P^\lambda(q,n) \, .
\end{equation}
We refer to $\widetilde{\ehr}_P^\lambda(x)$ as the \textit{$q$-Ehrhart polynomial} with respect to $\lambda$. 
We often denote the linear form $\lambda$ as a vector $(\lambda_1,\hdots,\lambda_d)\in\ZZ^d$, where $\lambda_j=\lambda(e_j)$. 

Parallel to the classical case, structural results for $\ehr_P^\lambda(q,n)$ follow from studying the \emph{$q$-Ehrhart series} 
\begin{equation}\label{eq:qehrseries}
  \Ehr_P^\lambda(q,z) \ := \ \sum_{n\geq0} \ehr_P^\lambda(q,n) \, z^n.
\end{equation}
Chapoton showed that~\eqref{eq:qehrseries} can be written as a rational function whose denominator consists of factors $1 - q^j z$, where $j = \lambda(v)$ for a vertex $v$ of $P$.
Furthermore, Chapoton proved the reciprocity theorem
\begin{equation}\label{eq:qehrrec}
(-1)^{\dim(P)}\, \widetilde{\ehr}^\lambda_P(q,[-n]_{q}) \ = \ \widetilde{\ehr}^\lambda_{P^\circ}\left(\tfrac 1 q, [n]_{\frac 1 q} \right),
\end{equation}
where $P^\circ$ denotes the (relative) interior of $P$.
The case $q=1$ in~\eqref{eq:qehrrec} recovers the classical Ehrhart--Macdonald reciprocity theorem~\cite{macdonald,CRT}.

Given a poset $\Pi=([d],\preceq)$, the \emph{order polytope} $\mathcal{O}(\Pi)$ is the lattice polytope
\[
\mathcal{O}(\Pi) \ := \ \left\{  (x_1,\hdots,x_d)\in[0,1]^d : \, x_i\leq x_j \text{ if }i\preceq j  \right\}.
\]
Order polytopes were introduced by Stanley~\cite{S86}; they contain much information about a given poset and have provided important examples in polyhedral geometry.

Since all vertices of $\mathcal{O}(\Pi)$ are $0/1$-vectors, $\Ehr^\lambda_{\mathcal{O}(\Pi)}(q,z)$ can be written as a rational function with factors $1 - q^j z$ in the denominator where $j$ is a sum of some of the entries of $\lambda$. 
Chapoton's genericity condition means for order polytopes that the coordinates of $\lambda$
are positive, which we assume from now on.
We also note the following corollary, which we record for future purposes.

\begin{lemma}\label{lem:ehrpolcoeff}
Let $\Lambda := \lambda_1 + \lambda_2 + \dots + \lambda_d$.
The coefficients of $[\Lambda]_q! \, \widetilde{\ehr}^\lambda_{\mathcal{O}(\Pi)}(x)$ are polynomials in~$q$.
\end{lemma}
 
Kim and Stanton~\cite[Corollary~9.7]{KimStanton} gave the following (equivalent) formulas for the case when $\lambda = \bone$:
\begin{align}
\Ehr^\bone_{\mathcal{O}(\Pi)}(q,z) \ &= \ 
\frac{\displaystyle \sum_{\sigma\in\mathcal{L}(\Pi)}
q^{\comaj(\sigma)}z^{\des(\sigma)}}{(1-z)(1-qz)\cdots(1-q^dz)} \nonumber \\
\ehr^\bone_{\mathcal{O}(\Pi)}(q,n)\ &= \sum_{\sigma\in \mathcal{L}(\Pi)}
q^{\comaj(\sigma)}\left[{n+d-\des(\sigma) \atop d}\right]_q, \label{eq:kimstanton}
\end{align}
where $\mathcal{L}(\Pi)$ is the set of linear extensions of $\Pi$ and, writing a given linear extension $\sigma$ as a permutation of $[d]$, and \footnote{Here we fix a natural labeling of $\Pi$, i.e., an order-preserving bijection $\Pi \to [d]$.
The permutation corresponding to a given linear extension $\sigma$ can be read off from this labeling. 
Unfortunately, there are two different (and conflicting) definitions of the comajor index in the literature: the one we use here, and the sum of the ascent positions.
}
\[\Des(\sigma) := \{ j : \, \sigma(j+1) < \sigma(j) \},\]
\[\des(\sigma) := |\Des(\sigma)|, \text{ and}\]
\[\comaj(\sigma) := \sum_{j\in\Des(\sigma)}(d-j).\]
See \cite[Chapter 6]{CRT} for details on the interplay of linear extensions of a poset, their descent statistics, and the arithmetic of order polytopes.

Given a graph $G$, let $\mathcal{A}(G)$ denote the set of acyclic orientations of $G$; each acyclic orientation $\rho$ naturally induces a poset, which we denote $\Pi_\rho$.
There is a well-known connection (essentially going back to~\cite{stanleyacyclic}) between the chromatic polynomial of a given graph $G$ and the Ehrhart polynomials of the order polytopes of the acyclic orientations of $G$. 
In the language of $q$-chromatic polynomials and $q$-Ehrhart polynomials, it reads as follows.

\begin{lemma}\label{lem:chiintermsofehr}
The $q$-chromatic polynomial with respect to $\lambda$ equals
\[
\chi_G^\lambda(q,n) \ = \sum_{\rho\in \mathcal{A}(G)}
\ehr_{\mathcal{O}(\Pi_\rho)^\circ}^\lambda(q,n+1) \, .
\]
\end{lemma}

\begin{proof}
We follow the philosophy of inside-out polytopes~\cite{iop}. 
Let $d = |V|$.
We may interpret each $n$-coloring of the vertices of $G$ as a lattice point in the $(n+1)$st dilate of the open unit cube $(0,1)^d$ (where the $j$th coordinate is the color of vertex $j$). 
Furthermore, every proper $n$-coloring of $[d]$ is a lattice point that is not contained in the graphical hyperplane arrangement
\begin{equation}\label{eq:hyparr}
\mathcal{H}_G \ := \ \left\{ x_i=x_j: \, ij \in E \right\}.
\end{equation}
The regions of $(0,1)^d\setminus \mathcal{H}_G$ are precisely the open order polytopes $\mathcal{O}(\Pi_\rho)^\circ$ for $\rho$ an acyclic orientation of $G$. 
That is, each proper coloring $c$ of $G$ induces an acyclic orientation $\rho_c$ of $G$, where the edge $ij$ is oriented from $i$ to $j$ if $c(i)<c(j)$ and from $j$ to $i$ if $c(j)<c(i)$.
Therefore,
\begin{align*}
\chi_G^\lambda(q,n)
  \ = \sum_{\substack{\text{proper colorings}\\ c:[d]\to[n]}} q^{\lambda_1c(1)+\cdots +\lambda_d c(d)}
  \ &=\sum_{\rho\in \mathcal{A}(G)} \ \sum_{\mathbf{c}\in (n+1)\mathcal{O}(\Pi_\rho)^\circ \cap \ZZ^d} q^{\lambda_1c_1+\cdots +\lambda_d c_d} \\
&=\sum_{\rho\in\mathcal{A}(G)} \ehr_{\mathcal{O}(\Pi_\rho)^\circ}^\lambda(q,n+1) \, . \qedhere
\end{align*}
\end{proof}

\begin{lemma}\label{lem:unique}
Suppose $f(x)$ and $g(x)$ are polynomials with coefficients that are rational functions in $q$ such that
\[
  f([n]_q) \ = \ g([n]_q)
  \qquad \text{ for all } \ n \in \ZZ_{ >0 } \, .
\]
Then $f(x) = g(x)$.
\end{lemma}

\begin{proof}
By our assumptions, the polynomial $f(x) - g(x) \in \QQ(q)[x]$ has infinitely many zeros and so must be the zero polynomial.
\end{proof}

Together with Chapoton's result~\eqref{eq:chapotonqehr},  Lemmas~\ref{lem:chiintermsofehr} and~\ref{lem:unique} prove Theorem~\ref{thm:chitildepoly}. 
In fact, we can see more, namely, that the analogue of Lemma~\ref{lem:ehrpolcoeff} holds also for $q$-chromatic polynomials.

\begin{corollary}
Let $\Lambda := \lambda_1 + \lambda_2 + \dots + \lambda_d$.
The coefficients of $[\Lambda]_q! \, \widetilde{\chi}^\lambda_G(q,x)$ are polynomials in~$q$.
\end{corollary}

In \cite[Section~4.1]{Chapoton}, Chapoton suggests to compute $\widetilde{\ehr}^\lambda_P (q, \frac{ 1 }{
1-q })$ (and does so in the case when $P$ is an order polytope), as an analogue of computing $\ehr_P^\lambda(q,t)$ in the limit as $t \to \infty$.

\begin{theorem}\label{thm:oneoveroneminusq}
For any graph $G=(V,E)$,
\[
\widetilde{\chi}_G^\lambda \left( q, \frac{ 1 }{ 1-q } \right) \ = \sum_{\substack{\text{proper colorings}\\
c\,:\,V\to \ZZ_{ > 0 } }} q^{ \sum_{ v \in V } \lambda_v c(v) }.
\]
Consequently,
\[
\widetilde{\chi}_G^\bone \left( q, \frac{ 1 }{ 1-q } \right) \ = \ X_G \left( q, q^2, q^3, \dots \right) .
\]
\end{theorem}

\begin{example}
Let $G=P_2$, the path with 2 vertices, $v_1$ and $v_2$, and take $\lambda =\bone$. 
By Example \ref{ex:path}, we have \[\widetilde{\chi}_{P_2}^\bone(q,x) = \frac{2q^2}{q+1}x^2 - \frac{2q^2}{q+1}x. \]
Evaluating at $x= \frac{1}{1-q}$, we obtain \[\widetilde{\chi}_{P_2}^\bone\left(q,\frac{1}{1-q}\right) =
\frac{2q^2}{q+1} \left( \frac{1}{(1-q)^2} - \frac{1}{1-q}\right) = \frac{2q^3}{(1-q)^2(1+q)} \, , \]
and a quick calculation shows that 
this equals \[\sum_{\substack{\text{proper colorings}\\
c\,:\,V\to \ZZ_{ > 0 } }} q^{c(v_1)+c(v_2)} \, , \]
confirming Theorem~\ref{thm:oneoveroneminusq} in this case.
\end{example}

\begin{proof}
By Lemma~\eqref{lem:chiintermsofehr},
\[
\widetilde{\chi}_G^\bone \left( q, \frac{ 1 }{ 1-q } \right) \ = \sum_{\rho\in \mathcal{A}(G)}
\widetilde{\ehr}_{\mathcal{O}(\Pi_\rho)^\circ}^\lambda\left( q, \frac{ 1 }{ 1-q } \right) .
\]
Chapoton \cite[Section~4.1]{Chapoton} observed that
$\widetilde{\ehr}_{\mathcal{O}(\Pi_\rho)^\circ}^\lambda( q, \frac{ 1 }{ 1-q } )$ is the generating
function of all strictly order-preserving maps $\Pi_\rho \to \ZZ_{ > 0 }$, weighted by~$\lambda$. Thus
\[
\widetilde{\chi}_G^\lambda \left( q, \frac{ 1 }{ 1-q } \right) \ = \sum_{\substack{\text{proper colorings}\\
c\,:\,V\to \ZZ_{ > 0 } }} q^{ \sum_{ v \in V } \lambda_v c(v) }. \qedhere
\]
\end{proof}

Returning to the computation of $\chi^\lambda_G(q,n)$, the case $\lambda=\bone$ is particularly nice because we can employ~\eqref{eq:kimstanton}.

\begin{corollary}\label{cor:chibyqbinomials}
For any graph $G=(V,E)$,
\[
\chi^\bone_G(q,n) \ = \sum_{\rho\in \mathcal{A}(G)} \ \sum_{\sigma\in \mathcal{L}(\Pi_\rho)}q^{\binom{d+1}{2}-\comaj(\sigma)}\left[ n+\des(\sigma) \atop d \right]_q \, . 
\]
\end{corollary} 

\begin{proof} 
We apply~\eqref{eq:qehrrec} and~\eqref{eq:kimstanton}: 
\[
\ehr^\bone_{\mathcal{O}(\Pi_\rho)^\circ}(q,n)
  \ = \ (-1)^d\ehr^\bone_{\mathcal{O}(\Pi_\rho)}(\tfrac 1q,-n)
  \ = \sum_{\sigma\in \mathcal{L}(\Pi_\rho)}q^{\binom{d+1}{2}-\comaj(\sigma)}\left[ n+\des(\sigma)-1 \atop d\right]_q,
\]
where $d = |V| = \dim(\mathcal{O}(\Pi_{\rho}))=d$,
and so Lemma~\ref{lem:chiintermsofehr} finishes the proof. 
\end{proof}

We record a few consequences of the last corollary.

\begin{corollary}\label{cor:chibyqbinomialsleadingcoeff}
Given a graph $G=(V,E)$ on $d$ vertices, the leading coefficient of $\widetilde\chi^\bone_G(q,x)$ equals
\[
  \frac{1}{[d]_q!} \sum_{\rho\in \mathcal{A}(G)} \ \sum_{\sigma\in \mathcal{L}(\Pi_\rho)} q^{d+\maj(\sigma)} . 
\]
\end{corollary}

\begin{proof}
The $q$-binomial coefficient in Corollary~\ref{cor:chibyqbinomials}
can be expressed in terms of the $q$-integers via
\begin{align*}
\left[ n+\des(\sigma) \atop d\right]_q&= \ \frac{[n+\des(\sigma)]_q\cdots[n]_q\cdots[n+\des(\sigma)-(d-1)]_q}{[d]_q!}\\
&= \ \frac{1}{[d]_q!} \left(q^{\des(\sigma)} [n]_q+[\des(\sigma)]_q \right) \cdots[n]_q\cdots\left(\frac{[n]_q-[\asc(\sigma)]_q}{q^{\asc(\sigma)}}\right)
\end{align*}
(since $(d-1)-\des(\sigma)=\asc(\sigma)$, the number of ascents of $\sigma$), which gives the
leading coefficient of $\widetilde\chi^\bone_G(q,x)$ as
\[
  \frac{1}{[d]_q!}\sum_{\rho\in \mathcal{A}(G)}\sum_{\sigma\in\mathcal{L}(\Pi_\rho)}q^{\binom{d+1}{2}+\binom{\des(\sigma)+1}{2}-\binom{\asc(\sigma)+1}{2}-\comaj(\sigma)} \, .
\]
Using the relation $\asc \sigma + \des \sigma = d-1$ again, the exponent simplifies to
\[
\binom{d+1}{2}+\binom{\des(\sigma)+1}{2}-\binom{\asc(\sigma)+1}{2}-\comaj(\sigma)
\ = \ d+\maj(\sigma) \, . \qedhere
\]
\end{proof}

\begin{corollary}
Let $G=([d],E)$ and express $\chi^\bone_G(q,n)$ in the form
\[
\chi^\bone_G(q,n) \ = \ \sum_{j\geq0} \beta_j(q) \left[ n+j \atop d \right]_q \, .
\]
\begin{enumerate}[itemsep=1ex]
\item Each $\beta_i(q)$ is a polynomial in $q$ with nonnegative coefficients.
\item $\beta_0(q)=|\mathcal{A}(G)| \, q^{\binom{d+1}{2}}$; in particular, if $G$ is a tree then $\beta_0(q)=2^{d-1} q^{\binom{d+1}{2}}$.
\item The largest value $i$ for which $\beta_i(q)\neq0$ is $d-\xi$ where $\xi$ is the chromatic number of $G$. 
Moreover, 
\[
\beta_{d-\xi}(q) \ =\sum_{\substack{\text{proper colorings}\\ c:V\to [\xi]}} q ^{\sum_{v\in V}c(v)}.
\]
\end{enumerate}
\end{corollary} 

\begin{remark}
Corollary~\ref{cor:chibyqbinomials} gives another way of realizing the largest value $j$ for which $\beta_j(q)\neq0$, namely, as the maximal number $m$ of descents in a linear extension of a poset induced by an acyclic orientation of $G$. 
Therefore, the chromatic number of $G$ is equal to $d-m$, which is one more than the minimal number of ascents in a linear extension of a poset induced by an acyclic orientation of $G$. 
This fact is known as the
Gallai--Hasse--Roy--Vitaver Theorem (see, e.g., \cite[Theorem~7.17]{chartrandzhang}).
\end{remark}

We also remark that $\beta_{d-\xi}(q)$ distinguishes between some trees as the next example illustrates.
 
\begin{example}
Let $T_1$ be the path of length 3 and let $T_2$ be the star with degree sequence $(3,1,1,1)$. 
We compute
\begin{align*}
\chi^\bone_{T_1}(q,n) \ &= \ 8q^{10}\left[{n \atop 4}\right]_q+(4q^{9}+6q^8+4q^7)\left[{n+1 \atop 4}\right]_q+2q^6\left[{n+2 \atop 4}\right]_q \\
\chi^\bone_{T_2}(q,n) \ &= \ 8q^{10}\left[{n \atop 4}\right]_q+(5q^9+4q^8+5q^7)\left[{n+1 \atop 4}\right]_q+(q^7+q^5)\left[{n+2 \atop 4}\right]_q.
\end{align*}
In particular, $\chi^\bone_{T_1}(q,2)=2q^6$ while $\chi^\bone_{T_2}(q,2)=q^7+q^5$. 
However, the coefficient $\beta_{d-2}(q)$ is not enough to distinguish all non-isomorphic trees on $d$ vertices.
\end{example}

\begin{remark}\label{rem:loebl}
Loebl~\cite{loebl} defines his $q$-chromatic function $M_q(G,n)$
using the color set $\{0,1,\dots,n-1\}$.
The shift $c(v)=s(v)+1$ from his color convention to ours gives
\[
  \chi_G^\bone(q,n)=q^{|V|}M_q(G,n).
\]
For a graph $G=(V,E)$ and $A\subseteq E$, let $C(A)$ denote the
set of connected components of the spanning subgraph $(V,A)$.
For $W\in C(A)$, write $|W|$ for its number of vertices.
Loebl's inclusion--exclusion formula~\cite[Theorem~1]{loebl}
therefore takes the following form with our convention:
\[
  \chi_G^\bone(q,n)
  =q^{|V|}\sum_{A\subseteq E}(-1)^{|A|}
    \prod_{W\in C(A)}[n]_{q^{|W|}}.
\]
Similarly, \cite[Observation~1]{loebl} gives, for the complete graph
$K_m$ on $m\le n$ vertices,
\[
  \chi_{K_m}^\bone(q,n)
  =m!\,q^{m(m+1)/2}\qbinom{n}{m}.
\]
For $m>n$, there are no proper $n$-colorings, and hence
$\chi_{K_m}^\bone(q,n)=0$.
\end{remark}



\section{The structure of $q$-chromatic polynomials}\label{sec:structure}

As with the classic chromatic polynomial, the $q$-chromatic polynomial satisfies a
deletion--contraction relation. Naturally, this is a special case of the deletion--contraction formula for Crew--Spirkl's  weighted version of the chromatic symmetric function~\cite[Lemma~2]{crewspirkl}.

\begin{theorem}\label{thm:deletion-contraction}
Suppose $G=([d],E)$ is a graph, $\lambda=(\lambda_1,\hdots,\lambda_d)\in\ZZ_{>0}^d$, and $e=12\in E$. Then 
\[
\chi_G^\lambda(q,n) \ = \ \chi_{G\setminus e}^\lambda(q,n)-\chi_{G/e}^{(\lambda_1+\lambda_2,\lambda_3,\hdots,\lambda_d)}(q,n) \, .
\]
\end{theorem}

\begin{proof}
As usual, we observe that the proper $n$-colorings of $G$ are precisely the proper $n$-colorings $c$ of $G\setminus e$ that satisfy the additional condition $c(1)\neq c(2)$. 
Therefore, we may count them by counting all proper $n$-colorings $c$ of $G\setminus e$ and then removing all such colorings $c$ for which $c(1)=c(2)$:
\begin{align*}
\chi_G^\lambda(q,n) \ &=\sum_{\substack{\text{proper colorings}\\ c:[d]\to [n]\text{ of }G}} q^{\lambda_1c(1)+\cdots+\lambda_dc(d)}\\
&=\sum_{\substack{\text{proper colorings}\\ c:[d]\to [n]\text{ of }G\setminus e}} q^{\lambda_1c(1)+\cdots+\lambda_dc(d)}-\sum_{\substack{\text{proper colorings}\\ c:[d]\to [n]\text{ of }G\setminus e\\ \text{where }c(1)=c(2)}} q^{\lambda_1c(1)+\cdots+\lambda_dc(d)}\\
&= \ \chi_{G\setminus e}^\lambda(q,n)-\chi_{G/e}^{(\lambda_1+\lambda_2,\lambda_3,\hdots,\lambda_d)}(q,n)
\, . \qedhere
\end{align*}
\end{proof}

We observe that a similar computation enables us to express any $q$-chromatic polynomial (for general $\lambda$ with positive entries) as a linear combination of $q$-chromatic polynomials with $\lambda = \bone$, via a repeated expansion--addition process as follows.
If $G=([d],E)$ is a graph and $\lambda=(\lambda_1,\hdots,\lambda_d)\in \ZZ_{ > 0 }^d$ with $\lambda_1\geq2$, split the vertex $1$ into two vertices $1'$ and $1''$ with weights $\lambda_1-1$ and $1$, respectively. 
Create the \textit{expansion graph $\exp(G,e)$ of $G$ at $1$} with vertex set $\{1',1'',2,\hdots,d\}$ and edge set 
\[
\{1'i, 1''i  : \, i\in \{2,\hdots,d\}\text{ such that } 1i\in E\}\cup \{ij :\,
i,j\in\{2,\hdots,d\}\text{ such that }ij\in E\} \, ,
\]
and let the \textit{addition graph $\add(G,e)$ of $G$ at $1$} be $\text{exp}(G,e)$ with an edge $e$ added between the new vertices $1'$ and $1''$. 
Then 
\[
\chi_G^\lambda(q,n) \ = \
\chi_{\text{exp}(G,e)}^{(\lambda_1-1,1,\lambda_2,\hdots,\lambda_d)}(q,n)-\chi_{\text{add}(G,e)}^{(\lambda_1-1,1,\lambda_2,\hdots,\lambda_d)}(q,n)
\, .
\]
By repeatedly applying this process, we obtain the following result:
\begin{theorem}\label{thm:repeated_application}
If $G=([d],E)$ is a graph and $\lambda=(\lambda_1,\hdots,\lambda_d)\in\ZZ_{>0}^d$, then there exist graphs $H_1,\hdots,H_\ell$ on $\lambda_1+\cdots+\lambda_d$ vertices and integers $k_1,\hdots,k_\ell$ such that 
\[
\chi_G^\lambda(q,n) \ = \ \sum_{i=1}^\ell k_i\, \chi^\bone_{H_i}(q,n) \, .
\]
\end{theorem}

Our next result extends Stanley's famous reciprocity theorem for the chromatic polynomial to the $q$-setting. 
A (not necessarily proper) coloring $c$ of a graph $G$ is \emph{compatible} with an acyclic orientation $\rho$ of $G$ if $c$ (weakly) increases along oriented edges. 
Stanley~\cite{stanleyacyclic} proved that $|\chi_G(-n)|$ equals the number of pairs of an $n$-coloring and a compatible acyclic orientation of $G$. 
In particular, $|\chi_G(-1)|$ equals the number of acyclic orientations of $G$.
This generalizes as follows.

\begin{theorem}\label{thm:qchromrec}
Given a graph $G=(V,E)$ and $\lambda \in \ZZ_{>0}^V$, let $\Lambda := \sum_{ v \in V }
\lambda_v$. 
Then
\[
(-1)^{ |V| } q^\Lambda \, \widetilde{\chi}^\lambda_G \left( \tfrac 1q, [-n]_{\frac 1q} \right) \ = \sum_{(c,\rho)} q^{\sum_{v\in V(G)} \lambda_v c(v)},
\]
where the sum is over all pairs of an $n$-coloring $c$ and a compatible acyclic orientation~$\rho$.
\end{theorem}

\begin{example}
For $\lambda=\bone$, the path on 2 vertices has $q$-chromatic polynomial
\[
\widetilde{\chi}^\bone_{P_2}(q,x) \ = \ \frac{2q^2x^2 - 2q^2x}{1+q} \, .
\]
Therefore,
\[
(-q)^2 \, \widetilde{\chi}^\bone_{P_2}(\tfrac 1q ,x)
 \ = \ q^2\, \frac{2q^{-2}x^2-2q^{-2}x}{1+q^{-1}}
 \ = \ \frac{2qx^2-2qx}{1+q}
\]
and so, e.g.,
\[
(-q)^2 \, \widetilde{\chi}^\bone_{P_2}(\tfrac 1q,-q-q^2)
\ = \ \frac{2q(-q-q^2)^2-2q(-q-q^2)}{1+q}
\ = \ 2q^4+2q^3+2q^2.
\]
Indeed, this sums $q^{\sum c(v)}$ for the six pairs of 2-colorings and compatible acyclic orientations.
\end{example}

\vspace{0.25cm} 

\begin{proof}[Proof of Theorem~\ref{thm:qchromrec}]
Let $d := |V|$.
We apply Chapoton's reciprocity result~\eqref{eq:qehrrec} to Lemma~\ref{lem:chiintermsofehr}:
\begin{align*}
(-1)^d \, \widetilde{\chi}^\lambda_G (q,[-n]_q)
\ &=\sum_{\rho\in \mathcal{A}(G)}(-1)^d \, \widetilde{\ehr}^\lambda_{\mathcal{O}(\Pi_\rho)^\circ}(q,[-n+1]_q)\\
&=\sum_{\rho\in \mathcal{A}(G)}(-1)^d \, \widetilde{\ehr}^\lambda_{\mathcal{O}(\Pi_\rho)} \left(\tfrac
1q,[n-1]_{\frac 1q}\right) .
\end{align*}
Therefore,
\begin{align}
(-1)^d \, \widetilde{\chi}^\lambda_G \left( \tfrac 1q, [-n]_{ \frac 1q} \right)
\ &=\sum_{\rho\in \mathcal{A}(G)}(-1)^d \, \widetilde{\ehr}^\lambda_{\mathcal{O}(\Pi_\rho)}(q,[n-1]_{q}) \nonumber \\
&=\sum_{\rho\in \mathcal{A}(G)}\ehr^\lambda_{\mathcal{O}(\Pi_\rho)}(q,n-1) \, . \label{eq:almostdone}
\end{align}
The integer lattice points in $(n-1)\mathcal{O}(\Pi_\rho)$ can be interpreted as colorings of $G$ using the color set
$\{0,1,\hdots,n-1\}$ that are compatible with $\rho$,
and so~\eqref{eq:almostdone} equals
\[
\sum_{(c,\rho)} q^{\sum_{v\in V(G)} \lambda_v (c(v)-1)}. \qedhere
\]
\end{proof}

We conclude this section with one more way of computing $q$-chromatic polynomials.
A \emph{flat} of a given graph $G = (V, E)$ is a subset $S \subseteq E$ such that for any edge $e\notin S$, the subgraph $(V,S)$ has strictly more connected components than $(V,S\cup\{e\})$. 
Geometrically, the intersection $H_S$ of the hyperplanes of the graphical arrangement $\mathcal{H}_G$ in
\eqref{eq:hyparr} corresponding to $S$ forms a flat of $\mathcal{H}_G$.
Let $P(S)$ be the collection of vertex sets of the connected components induced by $S$, and
for $W \subseteq V$ and $\lambda \in \ZZ_{ > 0 }^V$, let
\[
  \Lambda_W := \sum_{ v \in W } \lambda_v \, .
\]
The flats of $G$ form a poset (in fact, a lattice), whose M\"obius function helps us compute, again via inside-out polytopes~\cite{iop} (see also~\cite[Chapter~7]{CRT}), that
\begin{align*}
\chi^\lambda_G(q,n)
 \ = \sum_{\text{flats }S\subseteq E}\mu(\varnothing,S) \ehr^\lambda_{(0,1)^V \cap H_S} (q, n+1)
 \ &= \sum_{\text{flats }S\subseteq E}\mu(\varnothing,S) \prod_{C \in P(S)} q^{\Lambda_C}
[n]_{q^{\Lambda_C}} \\
   &= \ q^{ \Lambda_V } \sum_{\text{\rm flats }S\subseteq E}\mu(\varnothing,S) \prod_{C \in P(S)} [n]_{q^{\Lambda_C}} \, .
\end{align*}
In particular, for a tree $T=(V,E)$,
\[
\chi^\lambda_T(q,n) \ = \ q^{ \Lambda_V } \sum_{S\subseteq E}(-1)^{|S|} \prod_{ C \in P(S) } [n]_{q^{\Lambda_C}} \, .
\]
These formulas can be viewed as analogues of \cite[Theorem~2.5]{StanleyChromatic}, where Stanley proves an expression for the chromatic symmetric function in the power sum basis.

Next, we employ the following trick from~\cite{Chapoton}: for integers $n \ge 0$ and $k \ge 1$,
\[
  \left. \frac{1-(1+qx-x)^k}{1-q^k} \right|_{ x = [n]_q } = \ [n]_{ q^k } \, .
\]
This yields the following formulas for $q$-chromatic polynomials.

\begin{theorem}\label{thm:mobiusformula}
Given a graph $G=(V,E)$ and $\lambda \in \ZZ_{>0}^V$,
\[
\widetilde\chi^\lambda_G(q,x) \ = \ q^{ \Lambda_V } \sum_{\text{\rm flats }S\subseteq E}\mu(\varnothing,S) \prod_{C \in P(S)} \frac{1-(1+qx-x)^{\Lambda_C}}{1-q^{\Lambda_C}} \, . 
\]
In particular, for a tree $T=(V,E)$,
\[
\widetilde\chi^\lambda_T(q,x) \ = \ q^{ \Lambda_V } \sum_{S\subseteq E}(-1)^{|S|} \prod_{C \in P(S)} \frac{1-(1+qx-x)^{\Lambda_C}}{1-q^{\Lambda_C}} \, . 
\] 
\end{theorem}

\begin{remark}\label{rmk:linear-coeff}
In the following section, we will study the leading coefficient of this polynomial and see that it appears to distinguish trees. 
This is certainly not true of all other coefficients. 
For example, we can see that any tree on $d$ vertices with the same total vertex weight $\Lambda_V$ has the same linear coefficient (and the same constant 0, like the ordinary chromatic polynomial). 
Since 
\[
\frac{1-(1+(q-1)x)^{\Lambda_C}}{1-q^{\Lambda_C}} \ = \
\frac{-\Lambda_C(q-1)x-\binom{\Lambda_C}{2}(q-1)^2x^2-\cdots}{1-q^{\Lambda_C}} \, ,
\] 
the only linear terms of $\widetilde{\chi}$ come from edge subsets $S$ that result in 1 connected component; for trees $T$, the only such set is $S=E$. 
Thus, for a tree, the linear coefficient is determined only by $d$ and $\Lambda_V$.
\end{remark}

\begin{example}\label{ex:paths}
Theorem~\ref{thm:mobiusformula} suggests highly structured formulas for certain families of
graphs; we exercise this for the path $P_k$ on $k$ vertices when $\lambda = \bone$, in analogy with the chromatic symmetric function~\cite[Exercise~7.47(k)]{EC2}.
\[
\widetilde\chi^\bone_{ P_k }(q,x)
  \ = \ q^k \sum_{S\subseteq E}(-1)^{|S|} \prod_{C \in P(S)} \frac{1-(1+qx-x)^{|C|}}{1-q^{|C|}}
  \ =  \ (-q)^k \sum_{S\subseteq E} \prod_{C \in P(S)} \Phi(q,x, |C|),
\] 
where
\[
  \Phi(q,x,j) \ := \ - \frac{1-(1+qx-x)^{j}}{1-q^{j}} 
\]
and we used the fact that (for a tree) $|S| + |P(S)| = |E| + 1$.
The subsets of $E$ (for the path $P_k$) are in one-to-one correspondence with the compositions
(i.e., ordered partitions) of $k$, with parts given by the sizes of the sets in $P(S)$. Thus
\[
  \sum_{ k \ge 1 } \widetilde\chi^\bone_{ P_k }(q,x) \, t^k
  \ = \ \sum_{ k \ge 1 } \sum_{S\subseteq E} \prod_{C \in P(S)} \Phi(q,x, |C|) \, (-qt)^k
  \ = \ \sum_{ \mu } \prod_{m \in P(\mu)} \Phi(q,x, m) \, (-qt)^{ |\mu| }, 
\]
where the sum is over all nonempty compositions $\mu$, we collect the parts of $\mu$ in the multiset
$P(\mu)$, and $|\mu|$ is the sum of the parts of $\mu$.
\end{example}

\begin{example}\label{ex:stars}
The analogous computation for the star $S_{ k+1 }$ on $k+1$ vertices gives
\[
\widetilde\chi^\bone_{ S_{ k+1 } }(q,x)
  \ = \ (-q)^{k+1} \sum_{S\subseteq E} \prod_{C \in P(S)} \Phi(q,x, |C|)
  \ = \ (-q)^{k+1} \sum_{ j=0 }^{ k } \binom k j \Phi(q,x,j+1) \, (-x)^{ k-j } 
\]
and so
\begin{align*}
  \sum_{ k \ge 0 } \widetilde\chi^\bone_{ S_{ k+1 } }(q,x) \, t^{ k+1 }
  \ &= \ \sum_{ k \ge 0 } \sum_{ j=0 }^{ k } \binom k j \Phi(q,x,j+1) \, (-x)^{ k-j } \, (-qt)^{k+1} \\
    &= \ -qt \sum_{ j \ge 0 } \Phi(q,x,j+1) \, (-x)^{ -j } \sum_{ k \ge j } \binom k j (xqt)^k \\
    &= \ -qt \sum_{ j \ge 0 } \Phi(q,x,j+1) \, (-x)^{ -j } \, \frac{ (xqt)^j }{ (1-xqt)^{ j+1 }  } \\
    &= \ \sum_{ j \ge 0 } (-1)^{ j+1 }  \Phi(q,x,j+1) \left( \frac{ qt }{ 1-xqt } \right)^{ j+1 } .
\end{align*}
\end{example}


\section{The Leading Coefficient of a $q$-chromatic Polynomial}\label{sec:leadingcoeff}

We now focus our attention on the leading coefficient $c^{ \lambda }_T (q)$ of $\widetilde\chi^\lambda_T(q,x)$ stemming from Theorem~\ref{thm:mobiusformula}.

\begin{corollary}\label{cor:mobiusformula}
Given a tree $T=(V,E)$ and $\lambda \in \ZZ_{>0}^V$, the leading coefficient of $\widetilde\chi^\lambda_T(q,x)$ equals
\[
  c^{ \lambda }_T (q)
  \ = \ (-1)^{ |V| } (q^2-q)^{ \Lambda_V } \sum_{S\subseteq E} \prod_{ C \in P(S) } \frac{ 1 }{ 1 - q^{ \Lambda_C } } \, .
\]
\end{corollary}

In particular,
\[
  [\Lambda_V]_q! \ c^{ \lambda }_T (q)
  \ = \ q^{ \Lambda_V } (-1)^{ |V| + \Lambda_V } \sum_{S\subseteq E} (1-q)^{ \Lambda_V - \kappa(S) } \frac{ [\Lambda_V]_q! }{ \prod_{ C \in P(S) } [\Lambda_C]_q }
\]
(where $\kappa(S)$ denotes the number of components of the subgraph induced by $S$) is visibly a polynomial in $q$, as the fraction is a $q$-multinomial coefficient times a polynomial.

\begin{remark}
Deletion--contraction extends to $c^{ \lambda }_T (q)$, and we provide a formula here which might be helpful for computations. 
Let $l$ be a leaf of $T$ and
\begin{align*}
  A \ &:= \ \left\{ S \subset E : l \notin e \text{ for all } e \in S \right\} \\
  B \ &:= \ \left\{ S \subset E : l \in e \text{ for some } e \in S
\right\} . 
\end{align*}
Let $T' = (V', E')$ be the tree with $l$ deleted; we will denote the number of connected components induced by $S \subseteq E'$ by $\kappa'(S)$.
We further define $\lambda'$ to be the vector $\lambda$ with $l$th entry removed, and $\lambda^+$ to stem from $\lambda$ where we add $\lambda_l$ to the entry corresponding to the neighbor of $l$, with corresponding notation $\Lambda^+_W$ for $W \subseteq E'$. 
Then
\[
  \sum_{S \in A} (1-q)^{ \Lambda_V - \kappa(S) } \frac{ [\Lambda_V]_q! }{ \prod_{ C \in P(S) } [\Lambda_C]_q }
  \ = \ (1-q)^{ \lambda_l - 1 } \frac{ [\Lambda_V]_q! }{ [\lambda_l]_q \,
[\Lambda_{ V' }]_q! } \sum_{ S \subseteq E' } (1-q)^{ \Lambda_{V'} -
\kappa'(S) } \frac{ [\Lambda_{V'}]_q! }{ \prod_{ C \in P(S) } [\Lambda_C]_q }
\]
and
\[
  \sum_{S \in B} (1-q)^{ \Lambda_V - \kappa(S) } \frac{ [\Lambda_V]_q! }{ \prod_{ C \in P(S) } [\Lambda_C]_q }
  \ = \ \sum_{ S \subseteq E' } (1-q)^{ \Lambda^+_{V'} - \kappa'(S) } \frac{
[\Lambda^+_{V'}]_q! }{ \prod_{ C \in P(S) } [\Lambda^+_C]_q } \, .
\]
Thus,
\begin{align*}
  [\Lambda_V]_q! \ c^{ \lambda }_T (q)
  \ &= \ q^{ \Lambda_V } (-1)^{ |V| + \Lambda_V } (1-q)^{ \lambda_l - 1 }
\frac{ [\Lambda_V]_q! }{ [\lambda_l]_q \, [\Lambda_{ V' }]_q! } \sum_{ S
\subseteq E' } (1-q)^{ \Lambda_{V'} -
\kappa'(S) } \frac{ [\Lambda_{V'}]_q! }{ \prod_{ C \in P(S) } [\Lambda_C]_q } \\
  &\qquad + q^{ \Lambda_V } (-1)^{ |V| + \Lambda_V } \sum_{ S \subseteq E' } (1-q)^{ \Lambda^+_{V'} - \kappa'(S) } \frac{
[\Lambda^+_{V'}]_q! }{ \prod_{ C \in P(S) } [\Lambda^+_C]_q } \\
  \ &= \ q^{ \lambda_l } (q-1)^{ \lambda_l - 1 } \frac{ [\Lambda_V]_q! }{
[\lambda_l]_q \, [\Lambda_{ V' }]_q! } \left( [\Lambda_{V'}]_q! \ c^{ \lambda' }_{T'} (q) \right)
  - \left( [\Lambda^+_{V'}]_q! \ c^{ \lambda^+ }_{T'} (q) \right) . 
\end{align*}
Again, the fraction is a polynomial (via a $q$-binomial coefficient).
\end{remark}

We now further focus on the case $\lambda = \bone$.
Corollaries~\ref{cor:chibyqbinomialsleadingcoeff} and~\ref{cor:mobiusformula} give the following two (quite different) expressions for the leading coefficient.

\begin{corollary}\label{cor:c1T}
Given a tree $T=(V,E)$ on $d$ vertices, the leading coefficient of $\widetilde\chi^\bone_T(q,x)$ equals
\begin{align*}
  c^{ \bone }_T (q)
  \ &= \ (q-q^2)^d\sum_{S\subseteq E}\prod_{C \in P(S)}\frac{1}{1-q^{|C|}} \\
    &= \ \frac{1}{[d]_q!} \sum_{(\rho,\sigma)}q^{d+\maj(\sigma)},
\end{align*}
where the sum ranges over all pairs of (acyclic) orientations $\rho$ of $T$ and linear extensions $\sigma$ of the poset induced by~$\rho$.
\end{corollary}

In Corollary~\ref{cor:c1T}, the latter expression for $c^\bone_T(q)$ illustrates that $\frac{1}{q^d}\, [d]_q! \, c^\bone_T(q)$ is a polynomial in $q$ with nonnegative coefficients. 
We provide this expression for all non-isomorphic trees on $d=6$ vertices in Figure~\ref{fig:6-verts}.
\begin{figure}[ht!]
\begin{tabular}{ |m{5.5cm}|m{9cm}|}
\hline \vspace{0.25cm}
\centering{$\displaystyle T$} &\vspace{0.25cm} \hspace{3.5cm} $\displaystyle \frac{[d]_q!}{q^d}\, c^\bone_T(q)$ \vspace{0.25cm}    \\ \hline  & \\
\centering{\begin{tikzpicture}[scale=1,vertex/.style={circle, draw,inner sep=0cm,minimum size=0.12cm}]
    \node[vertex] (1) at (0,0){};
    \node[vertex] (2) at (1,0){};
    \node[vertex] (3) at (2,0){};
    \node[vertex] (4) at (3,0){};
    \node[vertex] (5) at (4,0){};
    \node[vertex] (6) at (5,0){};
    \draw (1) -- (2);
    \draw (2) -- (3);
    \draw (3) -- (4);
    \draw (4) -- (5);
    \draw (5) -- (6);
\end{tikzpicture}} & $2q^{12} + 8q^{11} + 18q^{10} + 36q^{9} + 62q^8 + 78q^7 +102q^6 + 102q^5 + 106q^4 + 80q^3 + 62q^2 + 32q + 32$ \vspace{0.5cm} \\ \hline & \\
\centering{\begin{tikzpicture}[scale=1,vertex/.style={circle, draw,inner sep=0cm,minimum size=0.12cm}]
    \node[vertex] (1) at (0.1,0.6){};
    \node[vertex] (2) at (0.1,-0.6){};
    \node[vertex] (3) at (1,0){};
    \node[vertex] (4) at (2,0){};
    \node[vertex] (5) at (3,0){};
    \node[vertex] (6) at (4,0){};
    \draw (1) -- (3);
    \draw (2) -- (3);
    \draw (3) -- (4);
    \draw (4) -- (5);
    \draw (5) -- (6);
\end{tikzpicture} }& $q^{13} + q^{12} + 10q^{11} + 16q^{10} + 41q^9 + 57q^8 + 81q^7 + 95q^6 + 108q^5 + 100q^4 + 83q^3 + 59q^2 + 36q + 32$ \\ \hline & \\
\centering{\begin{tikzpicture}[scale=1,vertex/.style={circle, draw,inner sep=0cm,minimum size=0.12cm}]
    \node[vertex] (1) at (0.1,0.6){};
    \node[vertex] (2) at (0.1,-0.6){};
    \node[vertex] (3) at (1,0){};
    \node[vertex] (4) at (2,0){};
    \node[vertex] (5) at (2.9,0.6){};
    \node[vertex] (6) at (2.9,-0.6){};
    \draw (1) -- (3);
    \draw (2) -- (3);
    \draw (3) -- (4);
    \draw (4) -- (5);
    \draw (4) -- (6);
\end{tikzpicture} }& $4q^{12} + 8q^{11} + 18q^{10} + 42q^9 + 58q^8 + 78q^7 + 92q^6 + 110q^5 + 98q^4 + 82q^3 + 58q^2 + 40q + 32$ \\ \hline & \\
\centering{\begin{tikzpicture}[scale=1,vertex/.style={circle, draw,inner sep=0cm,minimum size=0.12cm}]
    \node[vertex] (1) at (0,-0.9){};
    \node[vertex] (2) at (1,-0.6){};
    \node[vertex] (3) at (2,-0.3){};
    \node[vertex] (4) at (2,0.5){};
    \node[vertex] (5) at (3,-0.6){};
    \node[vertex] (6) at (4,-0.9){};
    \draw (1) -- (2);
    \draw (2) -- (3);
    \draw (3) -- (4);
    \draw (3) -- (5);
    \draw (5) -- (6);
\end{tikzpicture} } & $2q^{12} + 9q^{11} + 20q^{10} + 34q^9 + 65q^8 + 77q^7 + 96q^6 + 104q^5 + 107q^4 + 76q^3 + 62q^2 + 36q + 32$ \\ \hline & \\
\centering{\begin{tikzpicture}[scale=1,vertex/.style={circle, draw,inner sep=0cm,minimum size=0.12cm}]
    \node[vertex] (1) at (0,0){};
    \node[vertex] (2) at (1,0){};
    \node[vertex] (3) at (1,-1){};
    \node[vertex] (4) at (1,1){};
    \node[vertex] (5) at (2,0){};
    \node[vertex] (6) at (3,0){};
    \draw (1) -- (2);
    \draw (2) -- (3);
    \draw (2) -- (4);
    \draw (2) -- (5);
    \draw (5) -- (6);
\end{tikzpicture}} & $q^{13} + 3q^{12} + 11q^{11} + 18q^{10} + 39q^9 + 60q^8 + 78q^7 + 87q^6 + 110q^5 + 101q^4 + 79q^3 + 59q^2 + 42q + 32$ \\ \hline & \\
\centering{\begin{tikzpicture}[scale=1,vertex/.style={circle, draw,inner sep=0cm,minimum size=0.12cm}]
    \node[vertex] (1) at (0,0){};
    \node[vertex] (2) at (1,0){};
    \node[vertex] (3) at (.3,.95){};
    \node[vertex] (4) at (-0.8,0.6){};
    \node[vertex] (5) at (-0.8,-0.6){};
    \node[vertex] (6) at (.3,-.95){};
    \draw (1) -- (2);
    \draw (1) -- (3);
    \draw (1) -- (4);
    \draw (1) -- (5);
    \draw (1) -- (6);
\end{tikzpicture}} & $q^{14} + 9q^{12} + 9q^{11} + 20q^{10} + 39q^9 + 60q^8 + 72q^7 + 81q^6 + 112q^5 + 99q^4 + 79q^3 + 58q^2 + 49q + 32$ \\ \hline
\end{tabular}
\caption{A table of the leading coefficients of the $q$-chromatic polynomials of the
non-isomorphic trees on $d=6$ vertices.
}\label{fig:6-verts}
\end{figure}

\begin{example}
Continuing Example~\ref{ex:paths}, we return to the path $P_k$ on $k$ vertices.
Corollary~\ref{cor:c1T} gives 
\[
  c^{ \bone }_{ P_k } (q)
  \ = \ (q-q^2)^{ k } \sum_{S\subseteq E} \prod_{ C \in P(S) } \frac{ 1 }{ 1 - q^{ |C| } }
  \ = \ q^{ k } \sum_{S\subseteq E} \prod_{ C \in P(S) } \frac{ (1-q)^{ |C| } }{ 1 - q^{ |C| } }
\]
and thus
\[
  \sum_{ k \ge 1 } c^{ \bone }_{ P_k } (q) \, t^k
  \ = \ \sum_{ k \ge 1 } \sum_{S\subseteq E} \prod_{ C \in P(S) } \frac{ (1-q)^{ |C| } }{ 1 - q^{ |C| } }
\left( qt \right)^k
  \ = \ \sum_{ \mu } \prod_{m \in P(\mu)} \frac{ \left( (1-q)qt \right)^m }{ 1 - q^{ m } } \, ,
\]
where again the sum is over all nonempty compositions~$\mu$.

This can be rephrased in the language of the classical \emph{Lambert series}
\[
  L(q,x) \ := \ \sum_{ n \ge 1 } \frac{ x^n }{ 1-q^n } \, .
\]
Namely, summing over all nonempty compositions $\mu$, 
\[
  \sum_{ \mu } \prod_{m \in P(\mu)} \frac{ x^m }{ 1 - q^{ m } } \ = \ \frac{ 1 }{ 1 - L(q,x) } - 1
\]
yields
\begin{equation}\label{eq:c1aslambert}
  \sum_{ k \ge 1 } c^{ \bone }_{ P_k } (q) \, t^k
  \ = \ \frac{ 1 }{ 1 - L(q, (1-q)qt) } - 1 \, .
\end{equation}

This, in turn, gives a connection to a different kind of permutation statistics, via Shareshian--Wachs'
\emph{$q$-Eulerian polynomial}~\cite{shareshianwachs}
\[
  A_k(q,t) \ := \ \sum_{ \sigma \in S_k } q^{ \maj(\sigma) - \exc(\sigma) } t^{ \exc(\sigma) } ,
\]
with the accompanying $q$-exponential generating function
\[
  1 + \sum_{ k \ge 1 } A_k(q,t) \, \frac{ x^k }{ [k]_q! }
  \ = \ \frac{ (1-t) \exp_q(x) }{ \exp_q(tx) - t \exp_q(x) } 
\]
where $\exp_q(x) := \sum_{ k \ge 0 } \frac{ x^k }{ [k]_q! }$.
We will be interested in the case $t=1$, that is, the generating polynomial of the statistic $\maj(\sigma) -
\exc(\sigma)$, with $q$-exponential generating function 
\[
  1 + \sum_{ k \ge 1 } A_k(q,1) \, \frac{ x^k }{ [k]_q! }
  \ = \ \frac{ - \exp_q(x) }{ x \exp_q'(x) - \exp_q(x) } 
  \ = \ \frac{ 1 }{ 1 - x \, \frac{ \exp_q'(x) }{ \exp_q(x) } } \, .
\]
We claim that
\begin{equation}\label{eq:qexpclaim}
  L(q, (1-q)x) \ = \ x \, \frac{ \exp_q'(x) }{ \exp_q(x) }
\end{equation}
which proves, via Corollary~\ref{cor:c1T} and \eqref{eq:c1aslambert}, that
\[
  \sum_{(\rho,\sigma)}q^{\maj(\sigma)} \ = \ \sum_{ \sigma \in S_k } q^{ \maj(\sigma) - \exc(\sigma) }
\]
where the first sum ranges over all pairs of (acyclic) orientations $\rho$ of $P_k$ and linear extensions $\sigma$ of the poset induced by~$\rho$.
Via \eqref{eq:c1aslambert} and \eqref{eq:qexpclaim}, we can thus also express
\[
  \sum_{ k \ge 1 } c^{ \bone }_{ P_k } (q) \, \left( \frac t q \right)^k \ = \ \frac{ \exp_q(t) }{ \exp_q(t) - t\exp_q'(t) } \, .
\]

To show \eqref{eq:qexpclaim}, we use the $q$-binomial theorem:
\[
  \exp_q(x)
  \ = \ \sum_{ n \ge 0 } \frac{ ((1-q)x)^n }{ (1-q) \cdots (1-q^n) } 
  \ = \ \prod_{ k \ge 0 } \frac{ 1 }{ 1 - (1-q) x q^k } 
\]
and so
\begin{align*}
  x \, \frac{ \exp_q'(x) }{ \exp_q(x) }
  \ &= \ \sum_{ k \ge 0 } \frac{ (1-q) x q^k }{ 1 - (1-q) x q^k } 
  \  = \ \sum_{ k \ge 0 } \sum_{ j \ge 1 } \left( (1-q) x q^k \right)^j
  \  = \ \sum_{ j \ge 1 } \left( (1-q) x \right)^j \sum_{ k \ge 0 } q^{ kj } \\
  \ &= \ \sum_{ j \ge 1 } \frac{ \left( (1-q) x \right)^j }{ 1 - q^j }
  \  = \ L(q, (1-q)x) \, .
\end{align*}
\end{example}

\begin{example}
Continuing Example~\ref{ex:stars} along similar lines, we compute for the star
\[
  c^{ \bone }_{ S_{k+1} } (q)
  \ = \ (q-q^2)^{ k+1 } \sum_{ j=0 }^{ k } \binom k j \frac{ 1 }{ (1-q^{ j+1 }) (1-q)^{ k-j }  } 
  \ = \ q^{ k+1 } \sum_{ j=0 }^{ k } \binom k j \frac{ (1-q)^{ j+1 } }{ 1-q^{ j+1 } } 
\]
and so
\[
  \sum_{ k \ge 0 } c^{ \bone }_{ S_{ k+1 } } (q) \, t^{k+1}
  \ = \ \sum_{ j \ge 0 } \frac{ (1-q)^{ j+1 } }{ 1-q^{ j+1 } } \sum_{ k \ge j } \binom k j
(qt)^{ k+1 }
  \ = \ \sum_{ j \ge 1 } \frac{ 1 }{ 1-q^{ j } } \left( \frac{ qt(1-q) }{ 1-qt } \right)^{ j } ,
\]
another Lambert series.
\end{example}

\begin{remark}\label{rem:starleadingcoeff}
Corollary~\ref{cor:c1T} immediately distinguishes stars from all other trees: the
largest possible major index one can obtain from a tree is from the linear extension
$[1,d,d-1,\hdots,3,2]$, and the only tree that realizes this is the star (with acyclic orientation where all edges point out from center). 
Consequently, the degree of $[d]_q! \, c^{\bone }_T (q)$ for a star is strictly larger than that of any other tree with the same number of vertices.
This implies that the chromatic symmetric function also distinguishes stars, which is
known (see, e.g.,~\cite{gonzalesorellanatomba,MartinMorinWagner}).
\end{remark}


\section{$G$-partitions}\label{sec:gpartitions}

The second formula in Corollary~\ref{cor:c1T} is reminiscent of Stanley's
$P$-partitions~\cite{stanleythesis} and organically suggests an extension of that concept to graphs.
We first review the part of Stanley's theory that we will need.

Given a poset $\Pi=([d],\preceq)$, a \emph{strict $\Pi$-partition} of $n \in
\ZZ_{ >0 }$ is a tuple $\bm \in \ZZ_{ >0 }^d$, such that\footnote{
Our definition differs from Stanley's inequalities, but the methodology is the same.
}
\[
  \sum_{ j=1 }^d m_j = n
  \qquad \text{ and } \qquad
  m_j < m_k \ \text{ whenever } \ j \prec k \, .
\]
Let $p_\Pi(n)$ denote the number of strict $\Pi$-partitions of $n$, with accompanying
generating function \[P_\Pi(q) := \sum_{ n > 0 } p_\Pi(n) \, q^n.\]
Then by~\cite[Exercise~6.23]{CRT}, 
\begin{equation}\label{eq:ppartgenfct}
  P_\Pi(q)
  \ = \ \frac{ q^d \sum_{ \sigma \in \mathcal{L}(\Pi) } \prod_{ j \in \Asc \sigma } q^{ d-j } }{ (1-q) (1-q^2) \cdots (1-q^d) } 
  \ = \ \frac{ q^d \sum_{ \sigma \in \mathcal{L}(\Pi) } q^{ \maj \sigma\op } }{ (1-q) (1-q^2) \cdots (1-q^d) }, 
\end{equation}
where $\Asc \sigma$ denotes the ascent set of $\sigma$, and we define $\sigma\op(j) := \sigma(d+1-j)$.
Note that we compute ascents and descents as in Section~\ref{sec:qehrhart}:
we fix some natural labeling of $\Pi$, i.e., an order-preserving bijection $\Pi \to [d]$.
The permutation corresponding to a given linear extension $\sigma$ can be read off from this labeling.
Viewing a poset as an (acyclic) directed graph, the following definition gives the natural
analogue for an undirected graph.

Let $G = (V, E)$ be a graph.
A \emph{$G$-partition}\footnote{
We follow the (somewhat misleading) nomenclature of Stanley---in general, neither $P$- nor
$G$-partitions are partitions, rather they are \emph{compositions}, i.e., ordered partitions
of a given integer~$n$.
} 
of $n \in \ZZ_{ >0 }$ is a tuple $\bm \in \ZZ_{ >0 }^V$, such that
\[
  \sum_{ v \in V } m_v = n
  \qquad \text{ and } \qquad
  m_v \ne m_w \ \text{ whenever } \ vw \in E \, .
\]
Let $p_G(n)$ denote the number of $G$-partitions of $n$, with accompanying
generating function $P_G(q) := \sum_{ n > 0 } p_G(n) \, q^n$.

\begin{theorem}
Let $G$ be a graph on $d$ vertices. Then
\begin{align*}
  P_G(q)
  \ &= \ \frac{ q^d \sum_{(\rho,\sigma)} q^{ \maj(\sigma\op) } }{ (1-q) (1-q^2) \cdots (1-q^d) } \\
    &= \ \frac{ q^{ \binom{ d+1 } 2 } \sum_{(\rho,\sigma)} q^{ - \maj(\sigma) } }{ (1-q) (1-q^2) \cdots (1-q^d) } \, ,
\end{align*}
where each sum ranges over all pairs of acyclic orientations $\rho$ of $G$ and linear extensions $\sigma$ of the poset induced by~$\rho$.
\end{theorem}

\begin{proof}
Since every $G$-partition is a $\Pi_\rho$-partition for exactly one acyclic
orientation $\rho$ of $G$ (and, conversely, every $\Pi_\rho$-partition is a
$G$-partition),
\[
  p_G(n) \ = \sum_{\rho\in \mathcal{A}(G)} p_{ \Pi_\rho } (n)
\]
and so~\eqref{eq:ppartgenfct} gives the first formula:
\[
  P_G(q)
  \ = \ \frac{ q^d \sum_{\rho\in \mathcal{A}(G)} \sum_{ \sigma \in \mathcal{L}(\Pi_\rho) } q^{ \maj(\sigma\op) } }{ (1-q) (1-q^2) \cdots (1-q^d) } \, .
\]

To see the second formula, we note that each $\rho\in \mathcal{A}(G)$ has a partner
orientation $\overline\rho \in \mathcal{A}(G)$ in which the direction of each edge is
reversed. A linear extension $\sigma \in \mathcal{L}(\Pi_\rho)$ has the corresponding
linear extension $\overline\sigma \in \mathcal{L}(\Pi_{\overline\rho})$ defined via
\[
  \overline\sigma(j)
  \ := \ d+1 - \sigma(d+1-j)
  \  = \ d+1 - \sigma\op(j) \, .
\]
In particular, $j \in \Des(\sigma\op)$ if and only if $j \in \Asc(\overline\sigma)$, and so
\[
  \sum_{(\rho,\sigma)} q^{ \maj(\sigma\op) }
  \ = \ \sum_{(\overline\rho,\overline\sigma)} q^{ \binom d 2 - \maj( \overline\sigma) } \, .
\qedhere
\]
\end{proof}

This yields a third equation that can be added to the ones in Corollary~\ref{cor:c1T}.

\begin{corollary}\label{cor:leadingcoeffgpartition}
Given a graph $G=(V,E)$ on $d$ vertices, the leading coefficient of $\widetilde\chi^\bone_G(q,x)$ equals
\[
  c^{ \bone }_G (q) \ = \ \left(q^2 - q\right)^d P_G \left( \tfrac 1 q \right) .
\]
\end{corollary}

We can now see Remark~\ref{rem:starleadingcoeff} through this new lens:
e.g., the star graph on $d$ vertices is unique with $p_G(d+1) = 1$.
More generally, Corollary~\ref{cor:leadingcoeffgpartition} implies that Conjecture~\ref{conj:main} is equivalent to the following.

\begin{conjecture}
The $G$-partition function $p_G(n)$ distinguishes trees.
\end{conjecture}

We conclude by making note of the connection between $G$-partitions and the stable principal evaluation $X_G(q,q^2,q^3,\hdots)$ of the chromatic symmetric function.
Namely, from first principles we can see that
\[
P_G(q) \ = \ X_G(q,q^2,q^3,\hdots) \, .
\]
We note that this rational generating function is implicit already in Stanley's work; see,
e.g.,~\cite[end of Section~1]{stanleyacyclic}.
This yields one final equation for the leading coefficient that can be added to the ones in Corollary~\ref{cor:c1T}.

\begin{corollary}\label{cor:leadingtostable}
Given a graph $G=(V,E)$ on $d$ vertices, the leading coefficient of $\widetilde\chi^\bone_G(q,x)$ equals
\[
  c^{ \bone }_G (q) \ = \ \left(q^2 - q\right)^d X_G \left( \tfrac 1 q, \tfrac{ 1 }{ q^2 }, \tfrac{
1 }{ q^3 }, \dots \right) .
\]
\end{corollary}


\section{Open Questions}

A line of open questions emerges concerning the coefficients of the $q$-chromatic polynomial. The classical chromatic polynomial $\chi_G(n)$ is very well studied, and many of its coefficients have nice combinatorial interpretations. Can we generalize these to $\widetilde{\chi}^\bone_G(q,x)$? For example:

\begin{enumerate}
\item The second coefficient of $\chi_G(n)$ is (negative) the number of edges of $G$. Can we refine this to a $q$ version, i.e., does the second coefficient of $\widetilde{\chi}^\bone_G(q,x)$ count the number of edges of $G$, but graded by some property of the edges?
\item Can the same be done for the linear coefficient which, in the classical case, counts the number of acyclic orientations with a unique sink at one fixed vertex? (This is not interesting for trees by Remark~\ref{rmk:linear-coeff}, but could be interesting for general graphs.)
\item The coefficients of $\chi_G(n)$ are alternating. Can we show that the coefficients of
$\widetilde{\chi}^\bone_G(q,x)$ are ``strongly alternating,'' in the sense that the coefficient of $x^j$ in $[d]_q!\cdot
\widetilde{\chi}^\bone_G(q,x)$ is a polynomial in $q$ with either all nonnegative or all nonpositive coefficients (depending on the parity of $d-j$)?
\end{enumerate}

Finally, as we mentioned in the introduction, there are further structural results
and questions that stem from viewing $\chi_G^\lambda(q,n)$ as an evaluation of a (weighted) chromatic symmetric function. 
It is then natural to ask if there is anything to be gained by zeroing in on the polynomial $\widetilde{\chi}^\lambda_G(q,x)$; for example:

\begin{enumerate}
\item[(4)] Is there some (interesting) variant of the $(3+1)$-free Conjecture of Stanley and Stembridge \cite{StanleyStembridge} for $\widetilde{\chi}^\lambda_G(q,x)$?
\item[(5)] There exist variants of Whitney's Broken-Circuit Theorem for weighted chromatic symmetric functions; see~\cite[Lemma~3]{crewspirkl} and~\cite[Theorem~6.8]{grinberg}.
Do they give rise to a meaningful broken-circuit result for the coefficients of $\widetilde{\chi}^\lambda_G(q,x)$? 
\end{enumerate}


\section*{AI disclosure}

The authors did not use AI or LLM tools for any aspect of this research or the writing of this manuscript.
SageMath was used for computations.
 
 \section*{Declarations}
 \subsection*{Conflict of Interest} The authors have no relevant financial or non-financial interests to disclose
 
 \subsection*{Acknowledgements} We thank Farid Aliniaeifard, Jos\'e Aliste-Prieto, Ben Braun, Fr\'ed\'eric Chapoton,
Logan Crew, Serkan Ho\c sten, Brendan McKay, Alejandro Morales, Rosa Orellana, Jianping Pan, Bruce Sagan, John
Shareshian, Richard Stanley, Steph van Willigenburg, Ole Warnaar, Jos\'e Zamora, and Tom Zaslavsky for helpful discussions about this work and pointers to the literature.

\subsection*{Funding}
Andr\'es R.~Vindas-Mel\'endez is partially supported by the NSF under Award~DMS-2532321. 

\subsection*{Data Availability}
 Data sets generated for this project are available from the authors upon request.


\bibliographystyle{amsplain}
\bibliography{references}

\setlength{\parskip}{0cm} 

\end{document}